\documentclass{amsart}

\usepackage{amsmath}
\usepackage{amssymb}
\usepackage{amsthm}
\usepackage{graphics}

\newtheorem{thrm}{Theorem}

\newtheorem{lemm}[thrm]{Lemma}
\newtheorem{cor}[thrm]{Corollary}
\theoremstyle{definition}\newtheorem{defn}[thrm]{Definition}
\theoremstyle{remark}

\newcommand{\vect}[2]{\left[ \begin{array}{r}{#1}\\{#2}\end{array} \right]}
\newcommand{\mat}[4]{\left[ \begin{array}{rr}{#1}&{#2}\\{#3}&{#4}\end{array} \right]}

\numberwithin{equation}{section}

\begin{document}
\title{Radix Representations, Self-Affine Tiles, and Multivariable Wavelets}
\author{Eva Curry}
\address{Department of Mathematics and Statistics, Acadia University,
  Wolfville, Nova Scotia, Canad B4P 2R6}
\email{eva.curry@acadiau.ca}

\subjclass{Primary 52C22, 42C40; Secondary 11A63}
\date{February 28, 2005}
\keywords{Self-affine tiling, radix representation, multivariable wavelet, Haar-like wavelet, dilation matrix}

\begin{abstract}We investigate the connection between radix representations for $\mathbb{Z}^n$ and self-affine tilings of $\mathbb{R}^n$.  We apply our results to show that Haar-like multivariable wavelets exist for all dilation matrices that are sufficie
ntly large.\end{abstract}

\maketitle

\section{Introduction}\label{sec_intro}
We investigate the connections between radix representations for $\mathbb{Z}^n$, self-affine tilings of $\mathbb{R}^n$, and Haar-like scaling functions for multiresolution analyses and associated wavelet sets.

In a separate paper, we investigate the idea, also introduced by Jeong~\cite{jeo}, of radix representations for vectors in $\mathbb{Z}^n$, or general point lattices $\Gamma = M(\mathbb{Z}^n)$ ($M$ a nondegenerate $n \times n$ matrix)~\cite{cu1}.  We wish 
to consider expanding matrices which preserve $\Gamma$.  Without loss of generality, we may assume $\Gamma = \mathbb{Z}^n$.  A matrix which preserves $\mathbb{Z}^n$ must have integer entries.
\begin{defn}  A \emph{dilation matrix} for $\mathbb{Z}^n$ is an $n \times n$ matrix $A$ with integer entries, all of whose eigenvalues $\lambda$ satisfy $|\lambda| > 1$.
\end{defn}
Note that for a dilation matrix $A$, $q := |\det{A}|$ is an integer, with $q > 1$.  Then $\mathbb{Z}^n / A(\mathbb{Z}^n)$ has nontrivial cokernel.  Let $D$ be a complete set of coset representatives of $\mathbb{Z}^n / A(\mathbb{Z}^n)$.  We call the elemen
ts of $D$ \emph{digits}.  

We may associate a sequence of digits with each $x \in \mathbb{Z}^n$ by the Euclidean algorithm, as follows.  Each $x \in \mathbb{Z}^n$ is in a unique coset of $\mathbb{Z}^n / A(\mathbb{Z}^n)$, thus there exist unique $x_1 \in \mathbb{Z}^n$ and $r_0 \in D
$ such that
\[ x = Ax_1 + r_0. \]
Similarly, for each $x_j$, $j \geq 1$, there exist unique $x_{j+1} \in \mathbb{Z}^n$ and $r_{j} \in D$ such that
\[ x_j = Ax_{j+1} + r_{j}. \]
Formally, we write
\[ x \sim \sum_{j=0}^{\infty} A^j r_j. \]
If there exists a nonnegative integer $N$ such that $r_j = \mathbf{0}$ for all $j > N$, then the Euclidean algorithm terminates and we say that $x$ has a radix representation with radix $A$ and digit set $D$.
\begin{defn}  Let $A$ be a dilation matrix.  We say that the matrix $A$ \emph{yields a radix representation with digit set $D$} if for every $x \in \mathbb{Z}^n$ there exists a nonnegative integer $N = N(x)$ and a sequence of digits $d_0, d_1, \ldots, d_N
$ in $D$ such that
\[ x = \sum_{j=0}^{N} A^j d_j. \]
\end{defn}
That is, a dilation matrix $A$ yields a radix representation with digit set $D$ if \emph{every} $x \in \mathbb{Z}^n$ has a radix representation with radix $A$ and digit set $D$.

Let $A$ be a dilation matrix, and define
\[ \mu = \min{\{ \sigma:\ \mbox{$\sigma$ a singular value of $A$} \}}. \]
Let $F$ be a fundamental domain for $\mathbb{Z}^n$, centered at the origin,
\[ F = \left[ -\frac{1}{2}, \frac{1}{2} \right)^n. \]
In \cite{cu1}, we give the following two results about radix representations.

\begin{thrm}\label{rad_rep}
Let $A$ be an $n \times n$ dilation matrix.  If $\mu > 2$ then $A$ yields a radix representation of $\mathbb{Z}^n$ with digit set $D = A(F) \cap \mathbb{Z}^n$.
\end{thrm}

\begin{cor}\label{big_enough}  For every dilation matrix $A$, there exists a positive integer $\beta \geq 1$ such that $A^{\beta}$ yields a radix representation with digit set $D_{\beta} = A^{\beta}(F) \cap \mathbb{Z}^n$.
\end{cor}

\section{Radix Representations and Tilings}\label{sec_radix}

Radix representations are closely related to self-affine tilings of $\mathbb{R}^n$.
\begin{defn}  A measurable set $Q \subset \mathbb{R}^n$ gives a self-affine tiling of $\mathbb{R}^n$ under translation by $\mathbb{Z}^n$ if
\begin{enumerate}
\item $\cup_{k \in \mathbb{Z}^n} (Q + k) = \mathbb{R}^n$, and the intersection $(Q+k_1) \cap (Q+k_2)$ has measure zero for any two distinct $k_1, k_2 \in \mathbb{Z}^n$ (tiling); and
\item there is a collection of $q = |\det{A}|$ vectors $k_1, \ldots, k_q \in \mathbb{Z}^n$ that are distinct coset representatives of $\mathbb{Z}^n / A(\mathbb{Z}^n)$ such that
\[ A(Q) \simeq \cup_{i=1}^{q} (Q + k_i) \quad \mbox{(self-affine)}. \]
\end{enumerate}
\end{defn}

Set
\[ T = T(A,D) :=\{\xi \in \mathbb{R}^n:\  \xi = \sum_{j=1}^{\infty} A^{-j} d_j\} \]
with the digits $d_j \in D$ for some digit set $D$.  One can easily check that $T$ is a self-affine set.  We would like to be able to think of the elements of $T$ as the fractional parts of vectors in $\mathbb{R}^n$ in the same way that the fractional par
ts of real numbers lie in $[0,1]$.  This is an accurate interpretation if $T$ is congruent to $\mathbb{R}^n / \mathbb{Z}^n$.  Thought of another way, we would like $T$ to tile $\mathbb{R}^n$ under translation by $\mathbb{Z}^n$.

\begin{thrm}\label{radix<=>tiling&0}  Let $A$ be a dilation matrix, and let $D$ be a digit set for $A$.  Then $A$ yields a radix representation with digit set $D$ if and only if the set $T(A,D)$ tiles $\mathbb{R}^n$ under translation by $\mathbb{Z}^n$ and
 the origin $\mathbf{0}$ is in the interior of $T$.
\end{thrm}

We split the proof of this theorem into a few lemmas.

\begin{lemm}\label{radix=>tiling}
Let $A$ be a dilation matrix, and let $D$ be a digit set for $A$.  If $A$ yields a radix representation with a digit set $D$, then the set $T(A,D)$ tiles $\mathbb{R}^n$ under translation by $\mathbb{Z}^n$.
\end{lemm}

\begin{proof}  Applying Proposition 5.19 from \cite{woj} the dilation matrix $A$, digit set $D$, and corresponding set $T = T(A,D)$ satisfy
\begin{enumerate}
\item $T$ is a compact subset of $\mathbb{R}^n$;
\item $A(T) = \cup_{d \in D} (T+d)$;
\item $\cup_{x \in \mathbb{Z}^n} (T+x) = \mathbb{R}^n$; and
\item $T$ contains an open set.
\end{enumerate}
To show that $T$ tiles $\mathbb{R}^n$ under translation by $\mathbb{Z}^n$, we must show that
\[ m((T+x) \bigcap (T+y)) = 0\ \mbox{for all}\ x \neq y, x,y \in \mathbb{Z}^n \]
(where $m(\cdot)$ denotes Lebesgue measure).  We extend an idea of Lagarias and Wang (\cite{lw1}, p.31) to show that $m((T+x) \cap (T+y)) = 0$ for any distinct $x,y \in \mathbb{Z}^n$ for which there is a radix representation with radix $A$ and digit set $
D$.

Note that for any subset $Q$ of $\mathbb{R}^n$, $m(A(Q)) = q m(Q)$ (where $q = |\det{A}|$).  In particular,
\begin{align*}
q m(T) &= m(A(T)) = m\left(\bigcup_{d \in D} (T+d)\right)\ \mbox{(by property $2$ of $T$)}\\
       &\leq \sum_{d \in D} m(T+d) = \sum_{d \in D} m(T) = q m(T).
\end{align*}
Additionally, property $2$ implies that $A^{k+1}(T) = \cup_{d \in D} (A^k(T) + A^kd)$ for all $k \geq 0$.  Then
\[
q^{k+1} m(T) \leq \sum_{d \in D} m(A^k(T) + A^kd) = \sum_{d \in D} q^k m(T) = q^{k+1} m(T).
\]
Thus
\[ m\left((A^k(T) + A^kd_i) \bigcap (A^k(T) + A^kd_j)\right) = 0 \]
for all distinct $d_i, d_j \in D$.

Since $\mathbf{0} \in D$, $A^{k+1}(T) \supset A^{k}(T)$ for all $k \geq 0$.  So 
\[ (A^{k+1}(T) + A^{k+1}d) \supset (A^{k}(T) + A^{k+1}d) \]
for all $d \in D$.  Now consider $T+x$ for any $x \in \mathbb{Z}^n$.  By hypothesis,
\[ x = A^Nd_N + \sum_{j=0}^{N-1}A^jd_j \]
for some $N\geq 0$, with $d_N \neq \mathbf{0} \in D$ and the $d_j \in D$.  Then
\begin{align*}
T+x &= \{y \in \mathbb{R}^n:\ y = A^Nd_N + \sum_{j=0}^{N-1}A^jd_j + \sum_{j=-\infty}^{-1}A^jd_j,\ \mbox{all}\ d_j \in D\}\\
    &\subset A^{N-1}(T) + A^Nd_N \subset A^{N}(T) + A^Nd_N,
\end{align*}
and
\[ (T \bigcap (T+x)) \subseteq (A^{N}(T) \bigcap (A^{N}(T) + A^Nd_N)). \]
Then $d_N \neq \mathbf{0} \in D$ implies that
\[ 0 = m\left(A^N(T) \bigcap (A^N(T) + A^Nd_N)\right) \geq m\left(T \bigcap (T+x)\right). \]  \end{proof}

If the set $T = T(A,D)$ tiles $\mathbb{R}^n$, it is not necessarily true that $A$ yields a radix representation.  For example, in the case where $A = 2$, we can find a radix representation for all nonnegative integers with the digit set $D = \{ 0, 1\}$, o
r for all non-positive integers with the digit set $D = \{0, -1\}$, but we cannot represent all integers with a radix representation using any digit set~\cite{mat}.  Yet $T = [0,1]$ (with digit set $D = \{0, 1\}$), does tile $\mathbb{R}$ under translation
 by $\mathbb{Z}$.  Similarly, if $A$ is the twin dragon matrix \cite{woj},
\[ A = \mat{1}{1}{-1}{1} \]
then $A$ is a dilation matrix.  A digit set for $A$ is
\[ D = \{d_0 = \vect{0}{0}, d_1 = \vect{1}{0}\}, \]
and the set $T$ generated by the twin dragon matrix with this digit set tiles $\mathbb{R}^2$ under translation by $\mathbb{Z}^n$~\cite{woj}, yet $A$ does not yield a radix representation of $\mathbb{Z}^n$ (for example, the vector \[\vect{0}{-1}\] does not
 have a radix representation)~\cite{cu1}.
\begin{figure}[tb]
\begin{center}
\includegraphics{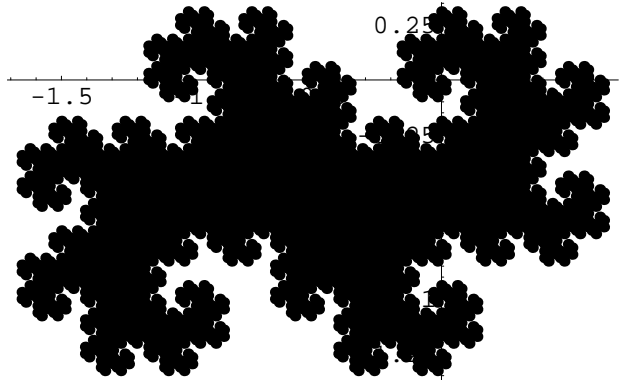}
\end{center}
\caption{The tile $T$ for the twin dragon matrix.}\label{fig:td_tile}
\end{figure}

In both of the examples, the origin $\mathbf{0}$ is on the boundary of the tile $T$, so that $\cup_{k =0}^{\infty} A^k(T) \subsetneq \mathbb{R}^n$.  We claim in Theorem \ref{radix<=>tiling&0} that this must be the case for all dilation matrices $A$ which 
give tiles $T$ but do not yield radix representations.  We first introduce a technical lemma.

\begin{lemm}\label{technical}  Let $A$ be a dilation matrix, and let $D$ be a digit set for $A$.  If $\mathbf{0} \not\in T^{\circ}$ then there exists an increasing subsequence $\{\ell_j\}_{j \geq 1}$ of the positive integers and a sequence of vectors $\{ 
\zeta_j:\ \zeta_j \in A^{-\ell_j}(\mathbb{Z}^n)\}_{j \geq 1}$ converging to $\mathbf{0}$ such that $\zeta_j \not\in T$ for all $j \geq 1$.
\end{lemm}

\begin{proof}  In the proof of Lemma \ref{radix=>tiling}, we noted that the set $T$ is compact.  Thus for any $\omega \in \mathbb{R}^n$ such that $\omega \not\in T$, $d(\omega,T) > 0$, and there exists an open ball centered at $\omega$, $B_{\omega}$, such
 that $\overline{B_{\omega}} \cap T = \emptyset$.

Let $\{y_j\}_{j \geq 1}$ be a sequence of vectors in $\mathbb{R}^n$ converging to $\mathbf{0}$ with $y_j \not\in T$ for all $j \geq 1$.  Set $\epsilon_j = \|y_j\|_{l^2}$, and notice that $\epsilon_j \geq d(y_j,T) > 0$, with $\lim_{j \rightarrow \infty} \epsilon_j = 0$.

Set $r_1 = \frac{d(y_1,T)}{2}$ and
\[ r_j = \min{\left\{ \frac{r_{j-1}}{2}, \frac{d(y_j,T)}{2} \right\}} \]
for $j \geq 2$.  Then $\{r_j\}_{j \geq 1}$ is a decreasing sequence of positive numbers, $\lim_{j \rightarrow \infty} r_j = 0$, and the open balls $B_j$ of radius $r_j$ centered at the vectors $y_j$ satisfy $\overline{B_j} \cap T = \emptyset$ for all $j \
geq 1$.  By construction, if $y_j^*$ is any point in the ball $B_j$ for each $j \geq 1$, then $\lim_{j \rightarrow \infty} y_j^* = \mathbf{0}$ and $y_j^* \not\in T$ for all $j \geq 1$.

In particular, there exists an increasing subsequence $\{\ell_j\}_{j \geq 1}$ of the positive integers such that $A^{-\ell_j}$ is a fine enough lattice to ensure that $A^{-\ell_j}(\mathbb{Z}^n) \cap B_j \neq \emptyset$ for each $j \geq 1$.  We may choose 
some $\zeta_j \in A^{-\ell_j}(\mathbb{Z}^n) \cap B_j$ for each $j \geq 1$.  By construction, $\zeta_j \not\in T$ for each $j \geq 1$, and the sequence $\{\zeta_j\}_{j \geq 1}$ converges to $\mathbf{0}$.
\end{proof}

\begin{lemm}\label{radix=>0}  Let $A$ be a dilation matrix, and let $D$ be a digit set for $A$.  If $A$ yields a radix representation with digit set $D$, then $\mathbf{0}$ is in the interior of the set $T$ generated by $A$ and $D$.
\end{lemm}

\begin{proof}  We prove the desired result by contradiction.  Assume that $\mathbf{0} \not\in T^{\circ}$.  By Lemma \ref{technical}, let $\{ \zeta_j \}_{j \geq 1}$ be a sequence of vectors in $\mathbb{R}^n$ converging to $\mathbf{0}$ such that $\zeta_j \i
n A^{-\ell_j}(\mathbb{Z}^n)$ (for some increasing subsequence $\{\ell_j\}_{j \geq 1}$ of the positive integers) but $\zeta_j \not\in T$ for all $j \geq 1$.  We may write $\zeta_j = A^{-\ell_j} x_j$ for some $x_j \in \mathbb{Z}^n$ for each $j$.  Since $A$ 
yields a radix representation, there exists an integer $N_j$ for each $x_j$ and digits $d^{(j)}_0, \ldots, d^{(j)}_{N_j} \in D$ such that
\[ x_j = \sum_{i=0}^{N_j} A^i d^{(j)}_i. \]
Thus
\[ \zeta_j = \sum_{i=0}^{N_j} A^{i-\ell_j} d^{(j)}_i = k_j + \xi_j \]
with $k_j \in \mathbb{Z}^n$ and $\xi_j \in T$.

Since $T$ is compact, $\| \xi_j \|_{l^2}$ is bounded above by $b$ for some $b > 0$.  Then, since $\zeta_j$ converges to $\mathbf{0}$, $\| k_j \|_{l^2}$ is also bounded above for sufficiently large $j$.  We use the rough estimate that there exists an integ
er $M \geq 1$ such that for all $j \geq M$, $\| k_j \|_{l^2} \leq 2b$.  Thus for $j \geq M$, the integer vectors $k_j$ all belong to a finite subset of $\mathbb{Z}^n$.  This implies that there exists an integer $N \geq 0$ such that $N_j \leq N$ for all $j
 \geq M$.  For $j > \max{\{M,N\}}$, $N-\ell_j \leq N-j < 0$, and thus $\zeta_j \in T$ for all sufficiently large $j$.  This contradicts the choice of $\zeta_j$, thus our assumption that $\mathbf{0} \not\in T^{\circ}$ must be false.
\end{proof}

We have shown that if a dilation matrix $A$ yields a radix representation with digit set $D$, then the set $T = T(A,D)$ tiles $\mathbb{R}^n$ under translation by $\mathbb{Z}^n$, and $\mathbf{0} \in T^{\circ}$.  Next we prove the converse.

\begin{lemm}\label{0&tiling=>radix}  Let $A$ be a dilation matrix, and let $D$ be a digit set for $A$.  If the set $T = T(A,D)$ tiles $\mathbb{R}^n$ under translation by $\mathbb{Z}^n$ and if $\mathbf{0} \in T^{\circ}$, then $A$ yields a radix representat
ion with digit set $D$.
\end{lemm}

\begin{proof}  Following the notation of \cite{lw1}, set
\[ D_{A,k} := \{x \in \mathbb{Z}^n:\ x = \sum_{j=0}^{k-1} A^j d_j, d_j \in D\},\]
the set of vectors that can be expressed with a radix representation of length less than or equal to $k$.  Note that $D_{A,1} = D$.  By construction, $A(T) = \cup_{d \in D} (T+d)$.  Thus
\[ A^{k}(T) = \bigcup_{x \in D_{A,k}} (T + x). \]
By the tiling hypothesis, $(T^{\circ}+x) \cap (T^{\circ}+y) = \emptyset$ for all distinct $x, y \in \mathbb{Z}^n$.  Thus $(T^{\circ} + y) \cap A^{k}(T^{\circ}) = \emptyset$ for all $y \in\mathbb{Z}^n$ with $y \not\in D_{A,k}$, and
\[ D_{A,k} \supseteq (A^k(T^{\circ}) \bigcap \mathbb{Z}^n). \]

Let $B$ be an open ball centered at the origin such that $B \subseteq T^{o}$.  Then
\[ D_{A,k} \supseteq (A^{k}(B) \cap \mathbb{Z}^n). \]
The sets $A^{k+1}(B)$ are expanding, with $\cup_{k \geq 0} A^{k+1}(B) = \mathbb{R}^n$.  Thus
\[ \bigcup_{k \geq 0} D_{A,k} \supseteq \bigcup_{k \geq 0} (A^k(B) \cap \mathbb{Z}^n) = \mathbb{Z}^n. \]
The opposite containment is true as well, since $D_{A,k} \subset \mathbb{Z}^n$ for each $k$.  Thus $\cup_{k \geq 0} D_{A,k} = \mathbb{Z}^n$.
\end{proof}

We have now completed the proof of Theorem \ref{radix<=>tiling&0}.  

Recall that $\mu$ is the smallest singular value of the dilation matrix $A$, and that $F$ is our canonical fundamental domain of $\mathbb{Z}^n$, $F = [-\frac{1}{2}, \frac{1}{2})^n$.  Combining Theorems \ref{rad_rep} and \ref{radix=>tiling}, we also have t
he following corollary.
\begin{cor}
Let $A$ be a dilation matrix, and let $D = A(F) \cap \mathbb{Z}^n$.  If $\mu > 2$ then the set
\[ T = \{ x \in \mathbb{R}^n:\ x = \sum_{j=-\infty}^{-1} A^j d_j, d_j \in D\}\]
tiles $\mathbb{R}^n$ under translation by $\mathbb{Z}^n$.
\end{cor}

\section{Haar-Like Wavelets}

Self-affine tiles allow us to construct multivariable wavelet sets associated with multiresolution analyses.  We review some basic definitions from wavelet theory here.
\begin{defn}
A \emph{multiresolution analysis (MRA)} associated with a dilation matrix $A$ is a nested sequences of subspaces $\cdots \subset V_{-1} \subset V_{0} \subset V_{1} \subset \cdots$ of $L^2(\mathbb{R}^n)$ satisfying:~\cite{woj}
\begin{enumerate}
\item $\overline{\cup_{j \in \mathbb{Z}} V_{j}} = L^2(\mathbb{R}^n)$;
\item $\cap_{j \in \mathbb{Z}} V_{j} = \{ 0 \}$;
\item $f(x) \in V_{j}$ if and only if $f(Ax) \in v_{j+1}$ for all $j \in \mathbb{Z}$;
\item $f(x) \in V_0$ if and only if $f(x-k) \in V_{0}$ for all $k \in \mathbb{Z}^n$; and
\item there exists a function $\phi(x) \in V_{0}$, called a \emph{scaling function}, such that
\[ \{\phi(x-k):\ k \in \mathbb{Z}^n\} \]
is a complete orthonormal basis for $V_{0}$.
\end{enumerate}
\end{defn}

The existence of multiresolution analyses in dimension $n>1$ has been studied by a number of authors.  In \cite{grm}, Gr\"{o}chenig and Madych showed that if $\phi = (m(Q))^{1/2} \chi_{Q}$ is a scaling function for a multiresolution analysis, then $Q$ mus
t be an affine image of a self-affine tiling of $\mathbb{R}^n$ under translation by $\mathbb{Z}^n$.  They showed also that if $T$ is a set of the form
\[ T = \{ x \in \mathbb{R}^n:\ x = \sum_{j=-\infty}^{-1} A^j d_j, d_j \in D\}\]
with $A$ a dilation matrix and $D$ a digit set for $A$, then $\phi = \chi_{T}$ is the scaling function for a multiresolution analysis (note that $|T| = 1$).  Lagarias and Wang noted that all self-affine tiles $T$ which tile $\mathbb{R}^n$ under translatio
n by $\mathbb{Z}^n$ must be of this form~\cite{lw2}.  A scaling function that is the characteristic function of some measurable set is called a Haar-like scaling function (after the Haar scaling function, which is $\chi_{[0,1]}$).  

Lagarias and Wang studied necessary conditions for sets of the form $T(A,D)$ to tile $\mathbb{R}^n$ under translation by $\mathbb{Z}^n$ in a series of papers (\cite{lw1}, \cite{lw4}, \cite{lw2}, \cite{lw3}).  Much of their work in these papers concerned t
he question of when a set $T(A,D)$ tiles $\mathbb{R}^n$ under translation by a sublattice of $\mathbb{Z}^n$.  In \cite{lw1} and \cite{lw4}, they showed that if $D$ is a complete set of coset representatives of $\mathbb{Z}^n / A(\mathbb{Z}^n)$, then $T(A,D
)$ is a self-affine tile of $\mathbb{R}^n$ under translation by some sublattice of $\mathbb{Z}^n$.  They also studied some properties of the tiling set $T(A,D)$.  He and Lau~\cite{hel} studied sets $T(A,D)$ which tile $\mathbb{R}^n$ under translation by a
 sublattice of $\mathbb{Z}^n$ as well; in particular, they looked at possible digit sets $D$.

In order for $\chi_{T}$ to be a scaling function for a multiresolution analysis, however, we need $T$ to tile $\mathbb{R}^n$ under translation by all of $\mathbb{Z}^n$.  Lagarias and Wang gave some necessary conditions for a dilation matrix $A$ to yield a
 Haar-like scaling function in \cite{lw2} and \cite{lw3}.  They showed that all dilation matrices in dimensions $n=2$ and $3$ yield Haar-like scaling functions.  Our results below give a sufficient condition for a dilation matrix $A$ to yield a Haar-like 
scaling function, in any dimension.  Note that for dilation matrices that yield a Haar-like scaling function, Strichartz has shown that multiresolution analyses and associated wavelet bases with arbitrary regularity can be constructed~\cite{str}.

The results of the previous section imply the following two theorems.
\begin{thrm}
Let $A$ be a dilation matrix, and let $D$ be a digit set for $A$ such that $A$ yields a radix representation for $\mathbb{Z}^n$ with digit set $D$.  Let $T$ be the set depending on $A$ and $D$ defined above.  Then $\phi = \chi_{T}$ is the scaling function
 for a multiresolution analysis.  In particular, if $A$ satisfies $\mu > 2$ and if $D$ is the set $D = A(F) \cap \mathbb{Z}^n$ with $F = [-\frac{1}{2},\frac{1}{2})^n$, then $\phi = \chi_{T}$ is the scaling function for a multiresolution analysis.
\end{thrm}

\begin{thrm}
Let $A$ be a dilation matrix.  Then there exists a positive integer $\beta \geq 1$ such that for all integers $k \geq \beta$ there exists a multiresolution analysis associated with the dilation matrix $A^k$.
\end{thrm}

Thus Haar-like scaling functions and associated MRAs exist for a large class of dilation matrices.



\end{document}